\documentclass[12pt, reqno, final]{amsart}

\usepackage{amsmath,amsthm}
\usepackage{amsfonts}  
\usepackage{amssymb}   
\usepackage{latexsym}
\usepackage{amscd}
\usepackage[mathscr]{eucal}
\usepackage[all]{xy}

\usepackage[pdftex]{graphicx}
\usepackage[bookmarksnumbered, colorlinks, plainpages]{hyperref}
\hypersetup{colorlinks=true,linkcolor=red, anchorcolor=green, citecolor=cyan, urlcolor=red, filecolor=magenta, pdftoolbar=true}

\DeclareMathAlphabet\mathoo{U}{eur}{b}{n}

 \DeclareMathOperator{\Real}{Re}
 \DeclareMathOperator{\Imaginary}{Im}
 \DeclareMathOperator{\Res}{Res}
 \DeclareMathOperator{\Ker}{Ker}
 \DeclareMathOperator{\Lin}{span}

\newtheorem{theorem}{Theorem}
\newtheorem{proposition}[theorem]{Proposition}
\newtheorem{corollary}[theorem]{Corollary}
\newtheorem{lemma}[theorem]{Lemma}
\theoremstyle{definition}
\newtheorem{definition}{Definition}
\newtheorem{example}{Example}
\theoremstyle{remark}
\newtheorem{remark}{Remark}

\begin{document}
\large

\title[A complete pair of solvents]{A complete pair of solvents\\ of a quadratic matrix pencil}

\author{V.~G. Kurbatov$^*$}
 \address{Department of System Analysis and Control,
Voronezh State University\\ 1, Universi\-tet\-skaya Square, Voronezh 394018, Russia}
\email{\textcolor[rgb]{0.00,0.00,0.84}{kv51@inbox.ru}}
\thanks{$^*$ Corresponding author}

\author{I.~V. Kurbatova}
 \address{Department of Software Development and Information Systems Administration,
Vo\-ro\-nezh State University\\ 1, Universitetskaya Square, Voronezh 394018, Russia}
\email{\textcolor[rgb]{0.00,0.00,0.84}{irakurbatova@gmail.com}}




\begin{abstract}
Let $B$ and $C$ be square complex matrices. The differential equation
\begin{equation*}
x''(t)+Bx'(t)+Cx(t)=f(t)
\end{equation*}
is considered.
A \emph{solvent} is a matrix solution $X$ of the equation $X^2+BX+C=\mathbf0$. A pair of solvents $X$ and $Z$ is called \emph{complete} if the matrix $X-Z$ is invertible. Knowing a complete pair of solvents $X$ and $Z$ allows us to reduce the solution of the initial value problem to the calculation of two matrix exponentials $e^{Xt}$ and $e^{Zt}$. The problem of finding a complete pair $X$ and $Z$, which leads to small rounding errors in solving the differential equation, is discussed.
\end{abstract}
\maketitle

\section*{Introduction}\label{s:intro}
The characteristic function of the second order linear differential equation
\begin{equation*}
x''(t)+Bx'(t)+Cx(t)=f(t)
\end{equation*}
with constant matrix coefficients $B$ and $C$ is the matrix pencil
\begin{equation*}
L(\lambda)=\lambda^2\mathbf1+\lambda B+C,\qquad \lambda\in\mathbb C.
\end{equation*}
A matrix solution of the equation
\begin{equation*}
X^2+BX+C=\mathbf0
\end{equation*}
is called~\cite[\S~3]{Krein-LangerI:eng} a (\emph{right}) \emph{solvent} of the pencil. They say~\cite[p. 75]{Gohberg-Lancaster-Rodman:Matrix_polynomials} (cf.~\cite[p.~392]{Krein-LangerI:eng}) that two right solvents $X$ and $Z$ form a \emph{complete pair} if the matrix $X-Z$ is invertible. If one has a complete pair of right solvents, the solution of the differential equation with the initial conditions
\begin{align*}
x(0)&=u_0,\\
x'(0)&=u_1
\end{align*}
can be represented (see Theorem~\ref{t:sol of ini val prob:2} Corollary~\ref{c:Perov} for more details) in the form
\begin{equation*}
x(t)=U'(t)u_0+U(t)(u_1+Bu_0)+\int_0^tU(t-s)f(s)\,ds,
\end{equation*}
where
\begin{align*}
U(t)&=\bigl(e^{Xt}-e^{Zt}\bigr)(X-Z)^{-1},\\
U'(t)&=\bigl(Xe^{Xt}-Ze^{Zt}\bigr)(X-Z)^{-1}.
\end{align*}
These formulas show that if a complete pair of right solvents is known, then solving the initial value problem is essentially reduced to the calculation of two matrix exponentials $e^{Xt}$ and $e^{Zt}$. This idea was first appeared in~\cite{Keldysh51:eng} and later was discussed by many authors, see, e.~g.,~\cite{Keldysh71:eng,Krein-LangerI:eng,Krein-LangerII:eng,Langer67,Langer76,Marcus88:eng,
Perov-Kostrub20:eng,Perov-Kostrub20a:eng,
Shkalikov-MS88:eng,Shkalikov96:eng,Shkalikov98,Shkalikov-Pliev89:eng}.

We consider matrices $B$ and $C$ of small size (about $10\times10$ or slightly larger). Only in this case, the representation of matrix exponentials as a function depending on $t$ is not very cumbersome and hence can be effectively used. On the other hand, in the same case the whole set of pairs of right solvents $X$ and $Z$ can be calculated; we describe a corresponding algorithm in Section~\ref{s:finding a comp pair}. The aim of the paper is to call attention to the fact that different pairs of right solvents $X$ and $Z$ are not interchangeable. The reason is that when constructing $U(t)$ we perform some linear algebra operations that may or may not be badly conditioned; for example, the condition number $\varkappa(X-Z)=\lVert X-Z\rVert\cdot\lVert (X-Z)^{-1}\rVert$ may be large, which leads to large errors when multiplying by $(X-Z)^{-1}$.

There are many papers devoted to numerical calculation of solvents, see, e.~g.,~\cite{Davis81,Guo-Higham-Tisseur08,Higham97,Higham-Kim00,Higham-Kim01,
Tisseur-Meerbergen01,Tsai-Chen-Shieh92}.
The present paper is distinguished by the formulation of the problem: we look for a complete pair of solvents that leads to small rounding errors.

The paper is organized as follows.
In Sections~\ref{s:square pencils}, we recall general information connected with quadratic pencils and their right solvents. In Section~\ref{s:Linearization}, we recall some facts related to Jordan chains of a quadratic pencil and its companion matrix. In Section~\ref{s:finding a comp pair}, we describe a spectral algorithm of finding a right solvent. In Sections~\ref{s:comp pair}, we describe how to find all complete pairs of right solvents and to choose the best one. In Section~\ref{s:num exper}, we present a few numerical examples.

\section{Quadratic pencils and right solvents}\label{s:square pencils}
In this section, we collect some general definitions and recall some general facts.

Let  $n\in\mathbb N$.
We denote by $\mathbb C^{n\times n}$ the linear space of all complex matrices of the size $n\times n$; the symbol $\mathbf1\in\mathbb C^{n\times n}$ denotes the identical matrix, and the symbol $\mathbf0\in\mathbb C^{n\times n}$ denotes the zero matrix.
Usually we do not distinguish matrices and the operators in $\mathbb C^n$ induced by them.

\begin{definition}
Let $n\in\mathbb N$ and $B,C\in\mathbb C^{n\times n}$. We consider the second order linear differential equation
\begin{equation}\label{e:DE 2nd order}
x''(t)+Bx'(t)+Cx(t)=f(t).
\end{equation}
The (\emph{quadratic}) \emph{pencil} corresponding to~\eqref{e:DE 2nd order} is the function
\begin{equation}\label{e:pencil 2nd order}
L(\lambda)=\lambda^2\mathbf1+\lambda B+C,\qquad \lambda\in\mathbb C,
\end{equation}
with values in $\mathbb C^{n\times n}$.
The \emph{resolvent} of pencil~\eqref{e:pencil 2nd order} is the function
\begin{equation*}
R_\lambda=(\lambda^2\mathbf1+\lambda B+C)^{-1},
\end{equation*}
and the \emph{resolvent set} of pencil~\eqref{e:pencil 2nd order} is the set $\rho(B,C)$, consisting of all $\lambda\in\mathbb C$ such that this inverse exists. The \emph{spectrum} of the pencil is the complement $\sigma(B,C)=\mathbb C\setminus\rho(B,C)$ of the resolvent set.

Similarly, $\rho(X)$ and $\sigma(X)$ respectively denote the resolvent set and the spectrum of a matrix $X\in\mathbb C^{n\times n}$.
\end{definition}

Let $U$ be an open set that contains $\sigma(B,C)$ and $f:\,U\to\mathbb C$ be a holomorphic function. We consider the matrices
\begin{align}
\psi(f)&=\frac1{2\pi i}\int_\Gamma f(\lambda)(\lambda^2\mathbf1+\lambda
B+C)^{-1}\,d\lambda,\label{e:psi for square pencil}\\
\chi(f)&=\frac1{2\pi i}\int_\Gamma \lambda f(\lambda)(\lambda^2\mathbf1+\lambda
B+C)^{-1}\,d\lambda,\label{e:chi for square pencil}
\end{align}
where $\Gamma$ is a contour laying in $U$, oriented counterclockwise, and surrounding $\sigma(B,C)$.

We set $U(t)=\psi(\exp_t)$, where $\exp_t(\lambda)=e^{\lambda t}$, or more detailed
\begin{equation}\label{e:U(t)}
U(t)=\frac1{2\pi i}\int_\Gamma e^{\lambda t}(\lambda^2\mathbf1+\lambda B+C)^{-1}\,d\lambda.
\end{equation}
Obviously, differentiating~\eqref{e:U(t)} with respect to $t$, we obtain
\begin{equation}\label{e:U'(t)}
U'(t)=\frac1{2\pi i}\int_\Gamma\lambda e^{\lambda t}(\lambda^2\mathbf1+\lambda
B+C)^{-1}\,d\lambda=\chi(\exp_t).
\end{equation}

\begin{proposition}\label{p:U(0)}
The following identities hold{\rm:}
\begin{equation*}
U(0)=\mathbf0,\qquad U'(0)=\mathbf1.
\end{equation*}
\end{proposition}
\begin{proof}
We have (the symbol $\Res$ means residue)
\begin{align*}
\lim_{\lambda\to\infty}R_\lambda&=
\lim_{\lambda\to\infty}\frac1{\lambda^2}\Bigl(\mathbf1+\frac1\lambda B+\frac1{\lambda^2}C\Bigr)^{-1}=\mathbf0,\\
\lim_{\lambda\to\infty}\lambda R_\lambda&=
\lim_{\lambda\to\infty}\frac1{\lambda}\Bigl(\mathbf1+\frac1\lambda B+\frac1{\lambda^2}C\Bigr)^{-1}=\mathbf0,\\
U(0)&=\frac1{2\pi i}\int_\Gamma(\lambda^2\mathbf1+\lambda B+C)^{-1}\,d\lambda\\
&=\Res_{\lambda=\infty}(\lambda^2\mathbf1+\lambda B+C)^{-1}\\
&=\lim_{\lambda\to\infty}\lambda(\lambda^2\mathbf1+\lambda B+C)^{-1}\\
&=\lim_{\lambda\to\infty}\frac1\lambda\Bigl(\mathbf1+\frac1\lambda B+\frac1{\lambda^2}C\Bigr)^{-1}=\mathbf0,\\
U'(0)&=\frac1{2\pi i}\int_\Gamma\lambda(\lambda^2\mathbf1+\lambda B+C)^{-1}\,d\lambda\\
&=\Res_{\lambda=\infty}\lambda(\lambda^2\mathbf1+\lambda B+C)^{-1}\\
&=\lim_{\lambda\to\infty}\lambda^2(\lambda^2\mathbf1+\lambda B+C)^{-1}\\
&=\lim_{\lambda\to\infty}\Bigl(\mathbf1+\frac1\lambda B+\frac1{\lambda^2}C\Bigr)^{-1}=\mathbf1.\qed
\end{align*}
\renewcommand\qed{}
\end{proof}

\begin{theorem}[{\rm see, e.g., \cite[Theorem 16]{Kurbatova-POMI:eng}}]
\label{t:sol of ini val prob:2}
Let $f:\,\mathbb R\to\mathbb C^n$ be a continuous function. Then the solution
$x$ of the initial value problem
\begin{equation}\label{e:ini value problem}
\begin{split}
x''(t)+Bx'(t)+Cx(t)&=f(t),\qquad t\in\mathbb R,\\
x(0)&=u_0,\\
x'(0)&=u_1
\end{split}
\end{equation}
and its derivative can be represented in the form
\begin{align}
x(t)&=U'(t)u_0+U(t)(u_1+Bu_0)+\int_0^tU(t-s)f(s)\,ds,&&t\in\mathbb R,\label{e:x(t):E is inv}\\
x'(t)&=U'(t)u_1-U(t)Cu_0+\int_0^tU'(t-s)f(s)\,ds,&&t\in\mathbb R.\notag
\end{align}
\end{theorem}

\begin{definition}
A matrix $X\in\mathbb C^{n\times n}$ is called~\cite[\S~3]{Krein-LangerI:eng} a (\emph{right}) \emph{solvent} of pencil~\eqref{e:pencil 2nd order} if $X$ satisfies the equation
\begin{equation*}
X^2+BX+C=\mathbf0.
\end{equation*}
A \emph{left solvent} is defined as a solution of the equation
\begin{equation*}
Y^2+YB+C=\mathbf0.
\end{equation*}
In this paper we deal only with right solvents.
\end{definition}
Not every quadratic pencil has at least one solvent, see Example~\ref{ex:pencil with Jord} below.

\begin{definition}
A \emph{factorization} of pencil~\eqref{e:pencil 2nd order} is its representation as the product of two linear pencils:
\begin{equation}\label{e:factorization}
\lambda^2\mathbf1+\lambda B+C=(\lambda\mathbf1-Y)(\lambda\mathbf1-X),
\end{equation}
where $X,Y\in\mathbb C^{n\times n}$.
\end{definition}

\begin{proposition}\label{p:spectrum factor}
Let pencil~\eqref{e:pencil 2nd order} possess factorization~\eqref{e:factorization}. Then
\begin{equation*}
\sigma(B,C)=\sigma(X)\cup\sigma(Y).
\end{equation*}
\end{proposition}
\begin{proof}
Let $\lambda\in\sigma(X)$. Then the kernel of the matrix $\lambda\mathbf1-X$ is non-zero. From formula~\eqref{e:factorization} it is seen that the kernel of the matrix $\lambda^2\mathbf1+\lambda B+C$ is also non-zero.

Now let $\lambda\in\sigma(Y)$. Then the image of the matrix $\lambda\mathbf1-Y$ does not coincide with the whole $\mathbb C^n$. Then from~\eqref{e:factorization} it is seen that the image of the matrix $\lambda^2\mathbf1+\lambda B+C$ also does not coincides with $\mathbb C^n$.

Finally, let $\lambda\in\sigma(B,C)$. Then the matrix $\lambda^2\mathbf1+\lambda B+C$ has a non-zero kernel. Let, for definiteness, a non-zero vector $h$ belong to $\Ker(\lambda^2\mathbf1+\lambda B+C)$. Then from~\eqref{e:factorization} it is seen that $(\lambda\mathbf1-Y)(\lambda\mathbf1-X)h=0$. Consequently, we have either $(\lambda\mathbf1-X)h=0$ or $(\lambda\mathbf1-Y)u=0$, where $u=(\lambda\mathbf1-X)h$ with $u\neq0$. In other words, either $\lambda\mathbf1-Y$ or $\lambda\mathbf1-X$ has a non-zero kernel.
\end{proof}

\begin{theorem}[{\rm\cite[p. 81]{Gantmakher59:eng},~\cite[p.~389]{Krein-LangerI:eng}}]\label{t:dual root}
Let $X$ be a right solvent of the equation
\begin{equation*}
X^2+BX+C=\mathbf 0.
\end{equation*}
Then
\begin{equation*}
\lambda^2\mathbf1+\lambda B+C=(\lambda\mathbf1+X+B)(\lambda\mathbf1-X).
\end{equation*}
\end{theorem}


\begin{definition}\label{def:separated roots}
A pair of right solvents $X$ and $Z$ is called \cite[p. 75]{Gohberg-Lancaster-Rodman:Matrix_polynomials} (cf.~\cite[p.~392]{Krein-LangerI:eng}) \emph{complete} if the matrix $X-Z$ is invertible.
\end{definition}

Even if a pencil has a right solvent, there is not always another right solvent $Z$ such that the pair $X$ and $Z$ is complete, see Example~\ref{ex:counterexample3} below.

Theorem~\ref{t:2.16} below and its Corollary~\ref{c:Perov} show that as soon as two right solvents form a complete pair, the calculation of matrices~\eqref{e:psi for square pencil} and~\eqref{e:chi for square pencil} is reduced to the calculation of two functions of linear pencils, which significantly simplifies the calculation of the matrices~\eqref{e:psi for square pencil} and~\eqref{e:chi for square pencil} or~\eqref{e:U(t)} and~\eqref{e:U'(t)}. This explains the importance of the problem of finding a complete pair.

\begin{lemma}\label{l:simp}
Let $X$ and $Z$ be arbitrary right solvents of pencil~\eqref{e:pencil 2nd order}. Then
\begin{equation}\label{e:(X^2-Z^2)+B(X-Z)=0}
(X^2-Z^2)+B(X-Z)=0.
\end{equation}
\end{lemma}
\begin{proof}
Subtracting $Z^2+BZ+C=\mathbf0$ from $X^2+BX+C=\mathbf0$, we arrive at~\eqref{e:(X^2-Z^2)+B(X-Z)=0}.
\end{proof}

\begin{proposition}[{\rm\cite[p. 393, Lemma 4.1]{Krein-LangerI:eng}}]\label{p:KL-L4.1}
Let $X$ and $Z$ be arbitrary right solvents of pencil~\eqref{e:pencil 2nd order}. Then for any $\lambda\in\rho(B,C)\cap\rho(X)\cap\rho(Z)$ the following identities hold{\rm:}
\begin{align*}
(\lambda^2\mathbf1+\lambda B+C)^{-1}(X-Z)
&=(\lambda\mathbf1-X)^{-1}(X-Z)(\lambda\mathbf1-Z)^{-1}\\
&=(\lambda\mathbf1-Z)^{-1}(X-Z)(\lambda\mathbf1-X)^{-1}\\
&=(\lambda\mathbf1-X)^{-1}-(\lambda\mathbf1-Z)^{-1}.
\end{align*}
\end{proposition}

\begin{proof}
By Theorem~\ref{t:dual root}, for $\lambda\in\rho(X)$ we have
\begin{align*}
(\lambda^2\mathbf1+\lambda B+C)(\lambda\mathbf1-X)^{-1}
=\lambda\mathbf1+X+B.
\end{align*}
Hence, taking into account equality~\eqref{e:(X^2-Z^2)+B(X-Z)=0}, we obtain that
\begin{align*}
(\lambda^2\mathbf1&+\lambda B+C)(\lambda\mathbf1-X)^{-1}(X-Z)
=(\lambda\mathbf1+X+B)(X-Z)\\
&=\lambda (X-Z)+X(X-Z)+B(X-Z)\\
&=\lambda (X-Z)+X(X-Z)+Z^2-X^2\\
&=\lambda (X-Z)+Z^2-XZ\\
&=\lambda (X-Z)-(X-Z)Z\\
&=(X-Z)(\lambda\mathbf1-Z).
\end{align*}
Multiplying the resulting equality by $(\lambda^2\mathbf1+\lambda B+C)^{-1}$ on the left and by $(\lambda\mathbf1-Z)^{-1}$ on the right, we obtain the first formula. The second formula follows from the first one by symmetry.

To prove the last formula, we note that
\begin{multline*}
(\lambda\mathbf1-Z)^{-1}(X-Z)(\lambda\mathbf1-X)^{-1}\\
=(\lambda\mathbf1-Z)^{-1}\bigl((\lambda\mathbf1-Z)-(\lambda\mathbf1-X)\bigr)(\lambda\mathbf1-X)^{-1}\\
=\bigl(\mathbf1-(\lambda\mathbf1-Z)^{-1}(\lambda\mathbf1-X)\bigr)(\lambda\mathbf1-X)^{-1}\\
=(\lambda\mathbf1-X)^{-1}-(\lambda\mathbf1-Z)^{-1}.\qed
\end{multline*}
\renewcommand\qed{}
\end{proof}

The following theorem shows that every complete pair of right solvents generates a factorization of the pencil. The converse is not always true, see Example~\ref{ex:counterexample3} below.
\begin{theorem}[{\rm\cite[p. 75, Lemma 2.14]{Gohberg-Lancaster-Rodman:Matrix_polynomials}, \cite[pp. 26--27]{Perov-Kostrub20a:eng}}]\label{t:Perov}
Let the right solvents $X$ and $Z$ of pencil~\eqref{e:pencil 2nd order} form a complete pair. Then
\begin{align}
\lambda^2\mathbf1+\lambda B+C&=(X-Z)(\lambda\mathbf1-X)(X-Z)^{-1}(\lambda\mathbf1-Z),\notag\\
(\lambda^2\mathbf1+\lambda B+C)^{-1}
&=\bigl((\lambda\mathbf1-X)^{-1}-(\lambda\mathbf1-Z)^{-1}\bigr)(X-Z)^{-1}.\label{e:factor:Perov}
\end{align}
\end{theorem}

The first formula gives rise to the factorization
\begin{equation}\label{e:factor}
\lambda^2 \mathbf1+\lambda B+C=(\lambda\mathbf1-Y)(\lambda\mathbf1-Z),
\end{equation}
where $Y=(X-Z)X(X-Z)^{-1}$.
The second one is an analogue of the partial fraction decomposition.

\begin{proof}
The both formulas follow from Proposition~\ref{p:KL-L4.1}.
\end{proof}

\begin{corollary}[{\rm\cite[p. 26]{Perov-Kostrub20a:eng}}]\label{c:rho of a square pencil}
Let the right solvents $X$ and $Z$ of pencil~\eqref{e:pencil 2nd order} form a complete pair. Then
\begin{equation*}
\rho(B,C)=\rho(X)\cap\rho(Z),\qquad
\sigma(B,C)=\sigma(X)\cup\sigma(Z).
\end{equation*}
\end{corollary}
\begin{proof}
By the De Morgan formulas, these identities are equivalent.

From Proposition~\ref{p:spectrum factor} and formula~\eqref{e:factor} it follows that
\begin{equation*}
\sigma(B,C)=\sigma(Z)\cup\sigma(Y),
\end{equation*}
where $Y=(X-Z)X(X-Z)^{-1}$. Since $Y$ and $X$ are similar, we have $\sigma(Y)=\sigma(X)$.
\end{proof}

\begin{theorem}[{\rm sufficient condition of completeness, \cite[Lemma 4.2]{Krein-LangerI:eng}, \cite[p. 76, Theorem 2.15]{Gohberg-Lancaster-Rodman:Matrix_polynomials}}]\label{t:2.15}
Let $X$ and $Z$ be arbitrary right solvents of pencil~\eqref{e:pencil 2nd order}. If $\sigma(X)\cap\sigma(Z)=\varnothing$, then the solvents $X$ and $Z$ form a complete pair.
\end{theorem}
\begin{proof}
Let us take an arbitrary vector $x\in\Ker(X-Z)$. By the formula
\begin{equation*}
(\lambda^2\mathbf1+\lambda B+C)^{-1}(X-Z)
=(\lambda\mathbf1-X)^{-1}-(\lambda\mathbf1-Z)^{-1}
\end{equation*}
from Proposition~\ref{p:KL-L4.1}, we have
\begin{equation*}
(\lambda\mathbf1-X)^{-1}x=(\lambda\mathbf1-Z)^{-1}x,\qquad \lambda\in\rho(X)\cap\rho(Z).
\end{equation*}
We consider the functions $h_1(\lambda)=(\lambda\mathbf1-X)^{-1}x$ and $h_2(\lambda)=(\lambda\mathbf1-Z)^{-1}x$ with the domains $\rho(X)$ and $\rho(Z)$  respectively. Since $\sigma(X)\cap\sigma(Z)=\varnothing$, the functions $h_1$ and $h_2$ are holomorphic continuations of each other to the whole complex plane. Moreover, they tend to zero at infinity. Consequently, by the Liouville theorem~\cite[ch.~I, \S~4.12]{Naimark72:eng}, the functions $h_1$ and $h_2$ are identically zero. But this implies that $x=0$. Hence the matrix $X-Z$ is invertible.
\end{proof}

The main application of complete pairs is described in Theorem~\ref{t:2.16}. It reduces the calculation of a function of a quadratic pencil to the calculation of the same function of two linear pencils, which essentially simplifies practical calculations.

\begin{theorem}[{\rm cf. \cite[p. 77, Theorem 2.16]{Gohberg-Lancaster-Rodman:Matrix_polynomials}, \cite[p. 28]{Perov-Kostrub20a:eng}}]\label{t:2.16}
Let right solvents $X$ and $Z$ form a complete pair for pencil~\eqref{e:pencil 2nd order}. Then 
\begin{align*}
\psi(f)
&=\biggl(\frac1{2\pi i}\int_\Gamma f(\lambda)(\lambda\mathbf1-X)^{-1}\,d\lambda\\
&-\frac1{2\pi i}\int_\Gamma f(\lambda)(\lambda\mathbf1-Z)^{-1}\,d\lambda\biggr)(X-Z)^{-1}.
\end{align*}
\end{theorem}
\begin{proof}
It follows from definition~\eqref{e:psi for square pencil} and formula~\eqref{e:factor:Perov}.
\end{proof}

\begin{corollary}\label{c:Perov}
Let right solvents $X$ and $Z$ form a complete pair of pencil~\eqref{e:pencil 2nd order}. Then for functions~\eqref{e:U(t)} and~\eqref{e:U'(t)} {\rm(}participated in the formulation of Theorem~\ref{t:sol of ini val prob:2}{\rm)} the following representations hold{\rm:}
\begin{align*}
U(t)&=\bigl(e^{Xt}-e^{Zt}\bigr)(X-Z)^{-1},\\
U'(t)&=\bigl(Xe^{Xt}-Ze^{Zt}\bigr)(X-Z)^{-1}.
\end{align*}
\end{corollary}
\begin{proof}
It follow from Theorems~\ref{t:sol of ini val prob:2} and~\ref{t:2.16}.
\end{proof}

\section{Jordan chains}\label{s:Linearization}

Let us consider matrix pencil~\eqref{e:pencil 2nd order}.
The block matrix
\begin{equation*}
\mathcal C_1=\begin{pmatrix}
\mathbf0 & \mathbf1 \\
-C & -B
\end{pmatrix}
\end{equation*}
is called~\cite[p.~14]{Gohberg-Lancaster-Rodman:Matrix_polynomials} the \emph{companion} matrix of pencil~\eqref{e:pencil 2nd order}.

 \begin{remark}\label{r:matrix representation}
It is well-know that differential equation~\eqref{e:DE 2nd order}
can be reduced to a system of two first order equations by the change $y_0=x$, $y_1=x'$:
\begin{align*}
y_0'(t)&=y_1(t),\\
y'_1(t)&=-Cy_0(t)-By_1(t)+f(t).
\end{align*}
The matrix $\mathcal C_1$ is exactly the matrix coefficient of this system.
 \end{remark}

\begin{proposition}[{\rm see also~\cite{Kostrub23:eng-rus}}]\label{p:R of C_1}
The resolvent of the matrix $\mathcal C_1$ is defined for all $\lambda\notin\sigma(B,C)$
and can be represented in the form
\begin{equation*}
(\lambda\mathbf1-\mathcal C_1)^{-1}
=(\lambda^2\mathbf1+\lambda B+C)^{-1}\begin{pmatrix}
\lambda\mathbf1+B & \mathbf1\\
-C & \lambda\mathbf1
\end{pmatrix}.
\end{equation*}
\end{proposition}
\begin{proof}
The equality (provided that $\lambda\notin\sigma(B,C)$)
\begin{equation*}
(\lambda^2\mathbf1+\lambda B+C)^{-1}
\begin{pmatrix}
\lambda\mathbf1+B & \mathbf1\\
-C & \lambda\mathbf1
\end{pmatrix}
\begin{pmatrix}
\lambda\mathbf1 & -\mathbf1 \\
C & \lambda\mathbf1+B
\end{pmatrix}=
\begin{pmatrix}
\mathbf1 & \mathbf0 \\
\mathbf0 & \mathbf1
\end{pmatrix}
\end{equation*}
and
\begin{equation*}
\begin{pmatrix}
\lambda\mathbf1 & \mathbf1 \\
C & \lambda\mathbf1+B
\end{pmatrix}(\lambda^2\mathbf1+\lambda B+C)^{-1}
\begin{pmatrix}
\lambda\mathbf1+B & -\mathbf1\\
-C & \lambda\mathbf1
\end{pmatrix}
=
\begin{pmatrix}
\mathbf1 & \mathbf0 \\
\mathbf0 & \mathbf1
\end{pmatrix}
\end{equation*}
are verified by straightforward multiplication.
\end{proof}

\begin{corollary}[{\rm\cite[p. 15, Proposition 1.3]{Gohberg-Lancaster-Rodman:Matrix_polynomials}}]\label{c:res of C_1}
For all $\lambda\notin\sigma(B,C)$, we have
\begin{equation*}
(\lambda^2\mathbf1+\lambda B+C)^{-1}=
\begin{pmatrix}\mathbf1&\mathbf0\end{pmatrix}
(\lambda\mathbf1-\mathcal C_1)^{-1}
\begin{pmatrix}\mathbf0\\\mathbf1\end{pmatrix}.
\end{equation*}
\end{corollary}
\begin{proof}
It follows from Proposition~\ref{p:R of C_1}.
\end{proof}

We recall~\cite{Michel-Herget93} that a sequence of vectors $h_0$, $h_1$, \dots, $h_k$ ($h_0\neq0$) is called a \emph{Jordan chain} of \emph{length} $k+1$ for a square matrix $\mathcal A$, \emph{corresponding to} $\lambda_0\in\mathbb C$, if it satisfies the relations
\begin{equation}\label{e:J-chain of a matrix}
\begin{split}
\mathcal Ah_0&=\lambda_0h_0,\\
\mathcal Ah_1&=\lambda_0h_1+h_0,\\
\dots&\dots\dots\dots\dots,\\
\mathcal Ah_k&=\lambda_0h_k+h_{k-1}.
\end{split}
\end{equation}
We say that a Jordan chain is \emph{maximal} if it can not be extended to a larger Jordan chain.

\begin{definition}\label{def:Jordan chain}
A sequence of vectors $x_0$, $x_1$, \dots, $x_k\in\mathbb C^n$ ($x_0\neq0$)
is called~\cite[p. 377]{Krein-LangerI:eng}, \cite[p. 24]{Gohberg-Lancaster-Rodman:Matrix_polynomials} a \emph{Jordan chain} of \emph{length} $k+1$ for pencil~\eqref{e:pencil 2nd order}, \emph{corresponding to} $\lambda_0\in\mathbb C$, if it satisfies the equalities
\begin{align*}
L(\lambda_0)x_0&=0,\\
L(\lambda_0)x_1+(2\lambda_0\mathbf1+B)x_0&=0,\\
L(\lambda_0)x_2+(2\lambda_0\mathbf1+B)x_1+x_0&=0,\\
L(\lambda_0)x_3+(2\lambda_0\mathbf1+B)x_2+x_1&=0,\\
\dots\dots\dots\dots\dots\dots\dots\dots\dots\dots\dots&\dots\\
L(\lambda_0)x_k+(2\lambda_0\mathbf1+B)x_{k-1}+x_{k-2}&=0.
\end{align*}
Setting $x_{-1}=x_{-2}=0$, these equalities can be rewritten in the uniform way:
\begin{equation}\label{e:Jordan chain for L}
L(\lambda_0)x_i+(2\lambda_0\mathbf1+B)x_{i-1}+x_{i-2}=0,\qquad i=0,1,\dots,k.
\end{equation}
The vector $x_0$ is called an \emph{eigenvector}, all others vectors $x_i$ are called \emph{generalized} or \emph{associated} \emph{eigenvectors}.
\end{definition}

\begin{remark}\label{r:rel to DE}
It is straightforward to verify~\cite[p. 377]{Krein-LangerI:eng} that vectors $x_0$, $x_1$, \dots, $x_k$ form a Jordan chain for pencil~\eqref{e:pencil 2nd order} if and only if the functions
\begin{equation*}
t\mapsto\Bigl(\frac{t^i}{i!}x_0+\frac{t^{i-1}}{(i-1)!}x_1+\ldots+x_i\Bigr)e^{\lambda_0 t},\qquad i=0,1,\dots,k,
\end{equation*}
satisfy the homogeneous differential equation $x''(t)+Bx'(t)+Cx(t)=0$.
\end{remark}

\begin{proposition}[{\rm\cite[p. 378]{Krein-LangerI:eng}}]\label{p:p. 292}
Vectors $x_0$, $x_1$, \dots, $x_k$ form a Jordan chain for pencil~\eqref{e:pencil 2nd order} if and only if the vectors
\begin{equation*}
\begin{pmatrix}
x_0 \\
\lambda_0 x_0 \\
\end{pmatrix},\quad
\begin{pmatrix}
x_1 \\
\lambda_0 x_1+x_0 \\
\end{pmatrix},\quad\dots,\quad
\begin{pmatrix}
x_i \\
\lambda_0 x_i+x_{i-1} \\
\end{pmatrix},\quad\dots,\quad
\begin{pmatrix}
x_k \\
\lambda_0 x_k+x_{k-1} \\
\end{pmatrix}
\end{equation*}
form a Jordan chain for the companion matrix $\mathcal C_1$.
\end{proposition}
\begin{proof}
The verification is straightforward.
\end{proof}

The following proposition shows that Jordan chains of a pencil can be found as Jordan chains of its right solvents (which is a standard problem of linear algebra).
\begin{proposition}[{\rm\cite[Lemma 3.1]{Krein-LangerI:eng}}]\label{p:kem 3.1}
Any Jordan chain $x_0$, $x_1$, \dots, $x_k$ of a right solvent $X$ of pencil~\eqref{e:pencil 2nd order} is simultaneously a Jordan chain for pencil~\eqref{e:pencil 2nd order}, corresponding to the same $\lambda_0$.
\end{proposition}
\begin{proof}
Let $x_0$, $x_1$, \dots, $x_k$ be a Jordan chain for a right solvent $X$. This means that $Xx_{i}=\lambda_0 x_{i}+x_{i-1}$, $i=0,1,2,\dots,k$, where $x_{-1}=0$. We rewrite the equality $Xx_{i}=\lambda_0 x_{i}+x_{i-1}$ in the form
\begin{equation*}
x_{i-1}=(X-\lambda_0\mathbf1)x_{i}.
\end{equation*}

It requires to verify~\eqref{e:Jordan chain for L}. We have (see Theorem~\ref{t:dual root})
\begin{multline*}
L(\lambda_0)x_i+(2\lambda_0\mathbf1+B)x_{i-1}+x_{i-2}\\
=L(\lambda_0)x_i+(2\lambda_0\mathbf1+B)x_{i-1}+(X-\lambda_0\mathbf1)x_{i-1}\\
=L(\lambda_0)x_i+(2\lambda_0\mathbf1+B+X-\lambda_0\mathbf1)x_{i-1}\\
=L(\lambda_0)x_i+(\lambda_0\mathbf1+B+X)x_{i-1}\\
=L(\lambda_0)x_i+(\lambda_0\mathbf1+B+X)(X-\lambda_0\mathbf1)x_{i}\\
=L(\lambda_0)x_i-L(\lambda_0)x_i=0.\qed
\end{multline*}
\renewcommand\qed{}
\end{proof}

\section{Finding a complete pair}\label{s:finding a comp pair}

We denote by $\Lin\mathcal A$ the subspace spanned by the columns of the matrix~$\mathcal A$.

\begin{theorem}[{\rm\cite[Lemma 5.1]{Krein-LangerII:eng}, \cite[p. 125, Theorem 4.6]{Gohberg-Lancaster-Rodman:Matrix_polynomials}}]\label{t:4.6}
A matrix $X$ is a right solvent of pencil~\eqref{e:pencil 2nd order} if and only if the subspace
\begin{equation*}
 \mathcal M=\Lin
\begin{pmatrix} \mathbf1\\X\end{pmatrix}
\end{equation*}
is invariant under the companion matrix $\mathcal C_1$.
\end{theorem}

\begin{proof}
Let $X$ be a right solvent. Then from the equality $X^2+BX+C=\mathbf0$ we have $-C-BX=X^2$. Therefore,
\begin{equation*}
\begin{pmatrix}
\mathbf0 & \mathbf1 \\ -C & -B
\end{pmatrix}
\begin{pmatrix} \mathbf1\\X\end{pmatrix}
=\begin{pmatrix} \mathbf1\\X\end{pmatrix}X,
\end{equation*}
which means that the columns of the matrix $\bigl(\begin{smallmatrix}\mathbf0 & \mathbf1 \\ -C & -B\end{smallmatrix}\bigr)
\bigl(\begin{smallmatrix}\mathbf1\\X\end{smallmatrix}\bigr)$ are linear combi\-nations of the columns of the matrix $\bigl(\begin{smallmatrix}\mathbf1\\X\end{smallmatrix}\bigr)$.

Conversely, let the columns of the matrix $\bigl(\begin{smallmatrix}\mathbf0 & \mathbf1 \\ -C & -B\end{smallmatrix}\bigr)
\bigl(\begin{smallmatrix}\mathbf1\\X\end{smallmatrix}\bigr)$ be linear combinations of the column of the matrix $\bigl(\begin{smallmatrix}\mathbf1\\X\end{smallmatrix}\bigr)$. This means that there exists a matrix $Q$ such that
\begin{equation*}
\begin{pmatrix}
\mathbf0 & \mathbf1 \\ -C & -B
\end{pmatrix}
\begin{pmatrix} \mathbf1\\X\end{pmatrix}
=\begin{pmatrix} \mathbf1\\X\end{pmatrix}Q.
\end{equation*}
It follows that $Q=X$ and $-C-BX=X^2$, which means that $X$ is a right solvent.
\end{proof}

\begin{proposition}\label{p:inv subspace}
Let $\mathcal A$ be a square matrix. A subspace $\mathcal M$ is invariant under $\mathcal A$ if and only if $\mathcal M$ is a linear span of some Jordan chains of the matrix~$\mathcal A$.
\end{proposition}
\begin{remark}\label{r:Jordan chains and inv subspacers}
Jordan chains in Proposition~\ref{p:inv subspace} must not be maximal, i.~e., $\mathcal M$ may be a linear span of $h_0$, \dots, $h_l$ from~\eqref{e:J-chain of a matrix} with $l<k$.
\end{remark}
\begin{proof}
If $\mathcal M$ is a linear span of some Jordan chains of the matrix $\mathcal A$, then, evidently, $\mathcal M$ is invariant under $\mathcal A$.
To prove the converse, we note that, by virtue of~\eqref{e:J-chain of a matrix}, Jordan chains of the restriction of $\mathcal A$ to $\mathcal M$ are Jordan chains of $\mathcal A$ itself.
\end{proof}

Theorem~\ref{t:4.6} allows us to propose an algorithm for finding a right solvent. Let we know a Jordan form for the companion matrix
\begin{equation*}
\mathcal C_1=\begin{pmatrix}
\mathbf0 & \mathbf1 \\
-C & -B
\end{pmatrix}.
\end{equation*}
So, for finding an invariant subspace of $\mathcal C_1$ of dimension $n$ one should compound different sets of generalized eigenvectors (if such a set contains a generalized eigenvector, it must contain all lower generalized vectors; but it must not contain the whole Jordan chain, see Remark~\ref{r:Jordan chains and inv subspacers}). We write the vectors forming this set as the columns of the block matrix (with blocks $X_1$ and $X_2$ of the size $n\times n$)
\begin{equation*}
\begin{pmatrix}
X_1 \\X_2
\end{pmatrix}.
\end{equation*}
Then we multiply this matrix on the right by $X_1^{-1}$ (provided that $X_1^{-1}$ exists) and arrive at the block matrix
\begin{equation*}
 \begin{pmatrix} \mathbf1\\X\end{pmatrix}=
 \begin{pmatrix}
   X_1X_1^{-1} \\
   X_2X_1^{-1}
 \end{pmatrix}.
\end{equation*}
We stress that the columns of $\bigl(\begin{smallmatrix}X_1 \\X_2\end{smallmatrix}\bigr)$ and $\bigl(\begin{smallmatrix}X_1X_1^{-1} \\X_2X_1^{-1}\end{smallmatrix}\bigr)$ span the same subspace.
By Theorem~\ref{t:4.6}, we obtain the right solvent
\begin{equation*}
X=X_2X_1^{-1}.
\end{equation*}

The following example shows that the matrix $X_1$ can be non-invertible. Clearly, in this case the linear span of the columns of the matrix $\bigl(\begin{smallmatrix}X_1 \\X_2\end{smallmatrix}\bigr)$ cannot be represented as a linear span of columns of a matrix of the form
$\bigl(\begin{smallmatrix}\mathbf1 \\X\end{smallmatrix}\bigr)$. Therefore, by Theorem~\ref{t:4.6}, the linear span of the columns of the matrix $\bigl(\begin{smallmatrix}X_1 \\X_2\end{smallmatrix}\bigr)$ does not generate a right solvent.

\begin{example}\label{ex:X1 is ininvert'}
Let $n=2$ and
\begin{equation*}
B=\begin{pmatrix}
1 & 0 \\
3 & 3
\end{pmatrix},\qquad
C=\begin{pmatrix}
1 & 0 \\
2 & 2
\end{pmatrix}.
\end{equation*}
Then the companion matrix $\mathcal C_1$ has the eigenvalues
\begin{equation*}
\lambda_1=-2,\qquad \lambda_2=-1,\qquad \lambda_{3,4}=\frac{-1\pm i\sqrt3}{2}.
\end{equation*}
We take for $\bigl(\begin{smallmatrix}X_1 \\X_2\end{smallmatrix}\bigr)$ the eigenvectors that correspond to $\lambda_1$ and $\lambda_2$:
\begin{equation*}
\begin{pmatrix}X_1 \\X_2\end{pmatrix}=
\begin{pmatrix}
0 & 0 \\
1 & 1 \\
0 & 0 \\
-2 & -1
\end{pmatrix}.
\end{equation*}
Clearly, the block $X_1$ is not invertible.
\end{example}

There are two possible problems with the solvent $X=X_2X_1^{-1}$. (i) The condition number
$\varkappa(X)=\lVert X_1\rVert\cdot\lVert X_1^{-1}\rVert$ can be large (i.~e., the situation is close to the non-invertibility of $X_1$). This may cause large errors in $X=X_2X_1^{-1}$.
(ii) Our ultimate goal in finding solvents is to use the formulas for $U(t)$ and $U'(t)$ from Corollary~\ref{c:Perov}. They contain $e^{Xt}$. The condition number of the operation $X\mapsto e^{Xt}$ is~\cite[\S~11.3.2]{Golub-Van_Loan96:eng}
\begin{equation*}
\nu(X,t)=\max_{\lVert H\rVert\le1}\biggl\lVert \int_0^te^{X(t-s)}He^{Xs}\,ds\biggr\rVert
\frac{\lVert X\rVert}{\lVert e^{Xt}\rVert}.
\end{equation*}
It is known~\cite[p. 979]{Van_Loan77}, \cite[p.~575]{Golub-Van_Loan96:eng} that $\nu(X,t)\ge t\lVert X\rVert$. Thus, if $\lVert X\rVert$ is large, the condition number $\nu(X,t)$ is also large, which can cause large errors in the calculation of $e^{Xt}$.

\section{Calculating complete pairs}\label{s:comp pair}

We need not a single right solvent, but a complete pair of right solvents. So, we proceed to the construction of a complete pair.

\begin{theorem}\label{t:2 roots}
Matrices $X,Z\in\mathbb C^{n\times n}$ form a complete pair of right solvents of pencil~\eqref{e:pencil 2nd order} if and only if the subspaces
\begin{equation*}
\mathcal M_1=\Lin\begin{pmatrix}\mathbf1\\X\end{pmatrix},\qquad
\mathcal M_2=\Lin\begin{pmatrix}\mathbf1\\Z\end{pmatrix}
\end{equation*}
are invariant under the companion matrix $\mathcal C_1$ and yield a decomposition of the space $\mathbb C^{2n}$ into the direct sum $\mathbb C^{2n}=\mathcal M_1\oplus\mathcal M_2$.
\end{theorem}
\begin{proof}
Let matrices $X$ and $Z$ form a complete pair of right solvents. Then it follows from Theorem~\ref{t:4.6} that the subspaces $\mathcal M_1$ and $\mathcal M_2$ are invariant under $\mathcal C_1$. We consider two cases.

The first case: let $\mathcal M_1\cap\mathcal M_2=\{0\}$. Then, by the dimension reasoning, we have the direct sum decomposition $\mathbb C^{2n}=\mathcal M_1\oplus\mathcal M_2$.

The second case: let $\mathcal M_1\cap\mathcal M_2\neq\{0\}$. Then the subspace $\mathcal M_1\cap\mathcal M_2\neq\{0\}$ is invariant under $\mathcal C_1$ and consequently contains a non-zero vector $h\in\mathbb C^{2n}$, which is an eigenvector for the matrix $\mathcal C_1$. We take $h_1,h_2\in\mathbb C^n$ such that
$\bigl(\begin{smallmatrix}\mathbf 1\\X\end{smallmatrix}\bigr)h_1=h$ and
$\bigl(\begin{smallmatrix}\mathbf 1\\Z\end{smallmatrix}\bigr)h_2=h$. Then
$\bigl(\begin{smallmatrix}\mathbf1&\mathbf1\\X&Z\end{smallmatrix}\bigr)
\bigl(\begin{smallmatrix}h_1\\-h_2\end{smallmatrix}\bigr)=0$. On the other hand,
$\bigl(\begin{smallmatrix}\mathbf1&\mathbf1\\X&Z\end{smallmatrix}\bigr)
\bigl(\begin{smallmatrix}h_1\\-h_2\end{smallmatrix}\bigr)=
\bigl(\begin{smallmatrix}h_1-h_2\\Xh_1-Zh_2\end{smallmatrix}\bigr)$. It follows that $h_1=h_2$ and $(X-Z)h_1=0$. This contradicts the invertibility of $X-Z$.

Conversely, let $\mathcal M_1$ and $\mathcal M_2$ be invariant under $\mathcal C_1$ and generate a decomposition of $\mathbb C^{2n}$ into a direct sum. We show that the matrix $X-Z$ is invertible. We suppose the contrary: let  $(X-Z)h=0$ for a non-zero vector $h\in\mathbb C^n$. Then
\begin{equation*}
\begin{pmatrix}
\mathbf1&\mathbf1\\
X&Z
\end{pmatrix}
\begin{pmatrix}
h \\
-h
\end{pmatrix}=0,
\end{equation*}
which implies that the rang of the matrix $\bigl(\begin{smallmatrix}\mathbf1&\mathbf1\\
X&Z\end{smallmatrix}\bigr)$ is less than $2n$, which contradicts the assumption $\mathcal M_1+\mathcal M_2=\mathbb C^{2n}$.
\end{proof}

\begin{proposition}\label{p:choice of M}
Let $\mathcal C\in\mathbb C^{2n\times 2n}$ be an arbitrary matrix, and let $\mathcal M_1,\mathcal M_2\subseteq\mathbb C^{2n}$ be two subspaces.

If the subspaces $\mathcal M_1$ and $\mathcal M_2$ are invariant under the matrix $\mathcal C$ and yield the decomposition of $\mathbb C^{2n}$ into the direct sum $\mathbb C^{2n}=\mathcal M_1\oplus\mathcal M_2$, then there exists a Jordan representation of the matrix $\mathcal C$ such that each its Jordan chain is contained either in $\mathcal M_1$ or in $\mathcal M_2$.

Conversely. If there exists a Jordan representation of the matrix $\mathcal C$ such that each Jordan chain {\rm(}corresponding to this Jordan representation{\rm)} is contained either in $\mathcal M_1$ or in $\mathcal M_2$, then the subspaces $\mathcal M_1$ and $\mathcal M_2$ are invariant under $\mathcal A$ and yield the decomposition of the space $\mathbb C^{2n}$ into the direct sum $\mathbb C^{2n}=\mathcal M_1\oplus\mathcal M_2$.
\end{proposition}
\begin{proof}
To prove the first statement, we construct the Jordan rep\-re\-sen\-tations for the restrictions of (the operator, generated by) $\mathcal C$ to the subspaces $\mathcal M_1$ and $\mathcal M_2$, and then take their direct sum. The second statement follows from Proposition~\ref{p:inv subspace}.
\end{proof}

Examples~\ref{ex:pencil with Jord} and~\ref{ex:counterexample3} below describe two situations when the construction of a complete pair of right solvents is impossible.

\begin{example}\label{ex:pencil with Jord}
We consider the pencil
\begin{equation*}
L(\lambda)=\lambda^2\mathbf1-J,
\end{equation*}
where $J$ is a Jordan block with zero eigenvalue. For example,
\begin{equation*}
L(\lambda)=\lambda^2\begin{pmatrix}
1 & 0 \\
0 & 1
\end{pmatrix}-
\begin{pmatrix}
0 & 1 \\
0 & 0
\end{pmatrix}.
\end{equation*}
It is known that this pencil has no right solvents. In fact, if we assume that $X$ is a right solvent, then by Theorem~\ref{t:dual root} we would have the factorization
\begin{equation*}
\lambda^2\mathbf1+J=(\lambda\mathbf1+X)(\lambda\mathbf1-X),
\end{equation*}
which shows that $X$ satisfies the equation
\begin{equation*}
X^2=J.
\end{equation*}
We show that this equation has no solutions. Actually, from the spectral mapping theorem~\cite[Theorem 10.28]{Rudin-FA73:eng} it follows that $\sigma(X)=\{0\}$; therefore $X$ is itself similar to a direct sum of two Jordan blocks of the size $1\times1$ with zero eigenvalues or one Jordan block of the size $2\times2$; in the both cases the matrix $X^2$ must have two linearly independent eigenvectors; but the Jordan block $J$ has only one linearly independent eigenvector.
\end{example}

\begin{example}\label{ex:counterexample3}
We consider the scalar differential equation
\begin{equation*}
x''(t)+2x'(t)+x(t)=f(t).
\end{equation*}
In this example, the matrices $B$ and $C$ are of the size $1\times1$, and the companion matrix is
\begin{equation*}
\mathcal C_1=\begin{pmatrix}
0 & 1 \\
-1 & -2
\end{pmatrix}.
\end{equation*}
It is easy to verify that the matrix $\mathcal C_1$ has the only eigenvalue $\lambda_1=-1$, and $\lambda_1$ corresponds to a Jordan block of the size $2\times2$. Therefore the space $\mathbb C^2$ can not be decomposed into a direct sum of two invariant subspaces $\mathcal M_1$ and $\mathcal M_2$ of dimension 1. By Theorem~\ref{t:2 roots}, this implies that it is impossible to construct a complete pair of right solvents. Nevertheless, one right solvent exists, it is $X=1$; therefore, by Theorem~\ref{t:dual root}, the pencil possesses a factorization. We stress, that the coefficients $B$ and $C$ are selfadjoint matrices (of the size $1\times1$). Thus, even a selfadjoint pencil can not have a complete pair of right solvents.
\end{example}

Theorems~\ref{t:4.6} and~\ref{t:2 roots} together with Propositions~\ref{p:inv subspace} and~\ref{p:choice of M} allows us to propose an algorithm of finding all complete pairs of right solvents. First we take a Jordan decomposition of the companion matrix $\mathcal C_1$, split the set of all Jordan chains into two parts so that the number of vectors in each part is equal to $n$, and then define $\mathcal M_1$ and $\mathcal M_2$ as spans of these parts. Then with the aid of Theorem~\ref{t:4.6} we construct right solvents $X$ and $Z$ corresponding to these subspaces. According to Theorem~\ref{t:2 roots} this pair of right solvents is complete.
By going through all possible splitting the set of all Jordan chains into two such parts, we obtain all complete pairs of right solvents $X$ and $Z$, corresponding to the considered Jordan representation of the matrix. We recommend assigning Jordan chains corresponding to the same eigenvalue (as well as to very close eigenvalues) to the same part, cf. Theorem~\ref{t:2.15}. Of course, not all complete pairs of right solvents $X$ and $Z$ will be written down but only those that correspond to the Jordan decomposition under consideration (for more details, see the discussion below and, in particular, Remark~\ref{r:double jordan block}).

\begin{remark}\label{r:double jordan block}
The set of Jordan chains (and their linear spans) is not always determined uniquely, although the Jordan form itself is uniquely determined (up to the order of Jordan blocks). For example, suppose that an eigenvalue $\lambda_0$ corresponds to two Jordan blocks of the same size $r\times r$. Then the Jordan chains (corresponding to $\lambda_0$) can be chosen in infinitely many different ways, since they are determined~\cite[\S~6.1]{Shilov-LA77:eng} by choosing a basis in the factor space $\Ker(\lambda_0\mathbf1-\mathcal C_1)^r/\Ker(\lambda_0\mathbf1-\mathcal C_1)^{r-1}$. From the point of view of the algorithm now discussed, this is not significant, since we recommend including into one part all Jordan chains corresponding to the same eigenvalue.
\end{remark}

Due to rounding errors, multiple eigenvalues are not encountered in practical calculations  (in other words, all roots of the characteristic polynomial are simple). As long as all Jordan blocks have the size $1\times1$, it is easy to see that there are $\binom{2n}{n}$ of splittings in total. Of course, listing all $\binom{2n}{n}$ splittings is only possible for small values of $n$; for example, $\binom{20}{10}=184756$,
$\binom{24}{12}=2704156$ (actually, these numbers should be divided by 2, because the pairs $X$, $Z$ and $Z$, $X$ are equivalent from our problem point of view).

We look over many complete pairs of right solvents, because some complete pairs may be inconvenient. The main application of a complete pair of right solvents is described in Corollary~\ref{c:Perov} or (if we need a function of a pencil other than the exponential one) Theorem~\ref{t:2.16}. These statements describe formulas, the implementation of which involves matrix operations. For example, they use the multiplication by the matrix $(X-Z)^{-1}$, which, by the definition of a complete pair, exists, but can be ill-conditioned, i.e. the number
\begin{equation}\label{e:kappa}
\varkappa(X-Z)=\lVert X-Z\rVert\cdot\lVert (X-Z)^{-1}\rVert
\end{equation}
can be large. We call the number $\varkappa$ the \emph{condition number of the pair}. More generally, the number
\begin{equation*}
\varkappa(A)=\lVert A\rVert\cdot\lVert A^{-1}\rVert
\end{equation*}
is called~\cite[p.~328]{Higham08} the \emph{condition number} of a matrix $A$.
The multiplication by a matrix with a large condition number leads to large rounding errors, which makes the formulas from Corollary~\ref{c:Perov} and Theorem~\ref{t:2.16} inaccurate.

So, we propose to find all complete pairs of right solvents (perhaps weeding out those of little use in advance, see Examples~\ref{ex:2} and~\ref{ex:3}), for each of them estimate possible rounding errors (see Section~\ref{s:num exper} for details) and then select the best pair.

Let us look a little more closely at what happens to multiple eigenvalues during calculations. Suppose that an eigenvalue $\lambda_0$ of the matrix $\mathcal C_1$ corresponds to an eigensubspace of dimension $m>1$; most likely, the action of $\mathcal C_1$ on this eigensubspace is described by a Jordan block (one or more). Due to the disturbance caused by rounding errors, the calculation results in $m$ different eigenvalues that are close to each other. In this case, eigenvectors corresponding to these eigenvalues can be also close to each other. If we assign them to different parts, then the condition number $\varkappa(X-Z)$ will be huge, making the corresponding complete pair $X$ and $Z$ unacceptable. It is for this reason that we recommend placing very close eigenvalues (for example, differing by $10^{-8}$ or less) into the same part in advance.

\section{Numerical experiments}\label{s:num exper}
In this section, we present three numerical examples.
We use `Wolfram Mathema\-tica'~\cite{Wolfram} in our numerical experiments.

\begin{figure}[thb]
\begin{center}
\includegraphics[width=0.45\textwidth]{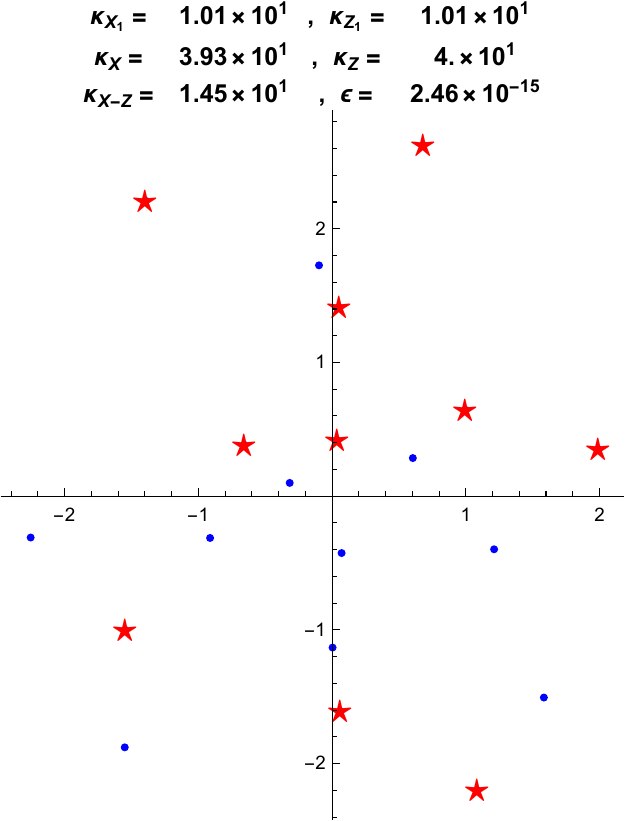}\hfill
\includegraphics[width=0.45\textwidth]{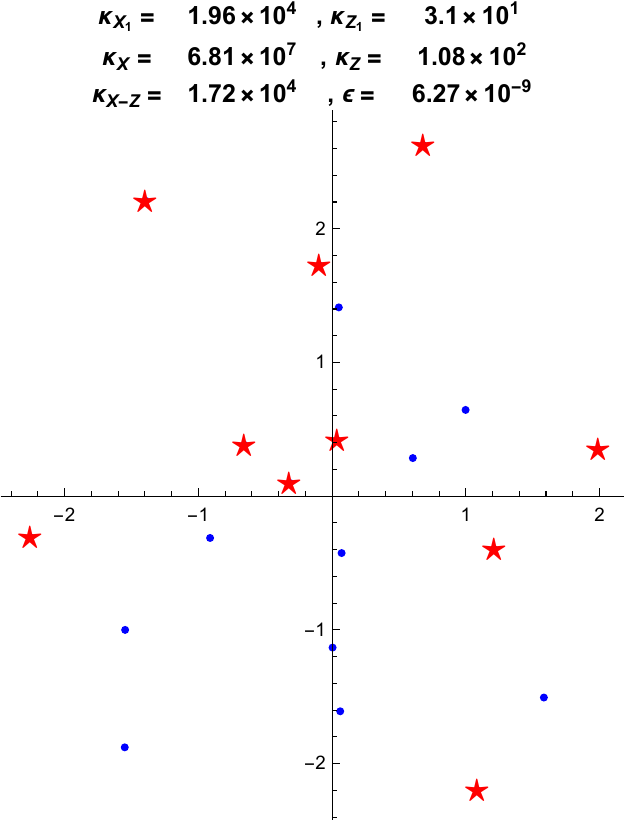}
 \caption{The best (left) and the worst (right) splittings of the spectrum of $\mathcal C_1$ from Example~\ref{ex:1}(a). The points of the spectrum of $\mathcal C_1$ related to $X$ are depicted by points, and related to $Z$ are depicted by stars} \label{f:comp1-1}
\end{center}
%
\begin{center}
\includegraphics[width=0.45\textwidth]{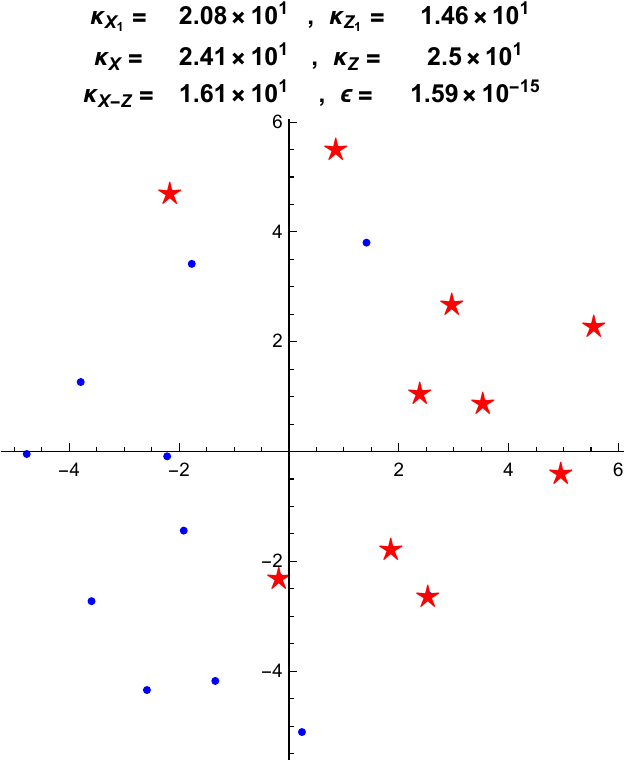}\hfill
\includegraphics[width=0.45\textwidth]{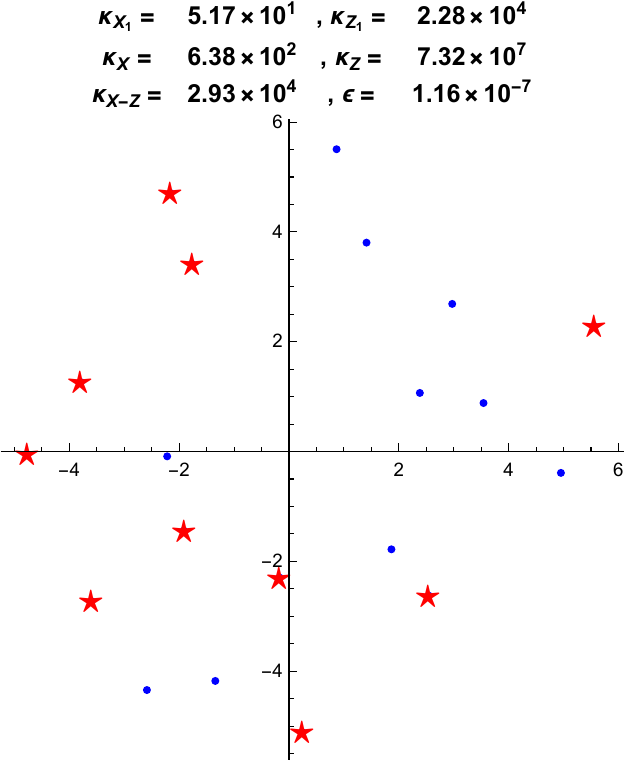}
 \caption{The best (left) and the worst (right) splittings of the spectrum of $\mathcal C_1$ from Example~\ref{ex:1}(b). The points of the spectrum of $\mathcal C_1$ related to $X$ are depicted by points, and related to $Z$ are depicted by stars} \label{f:comp1-10}
\end{center}
\end{figure}

\begin{figure}[thb]
\begin{center}
\includegraphics[width=0.45\textwidth]{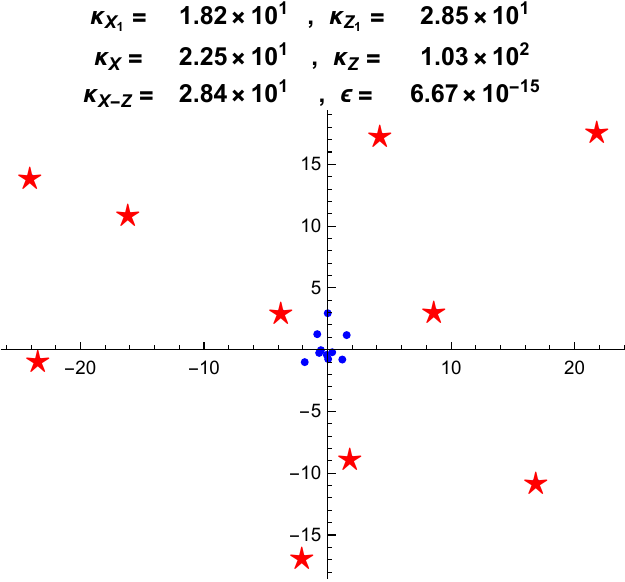}\hfill
\includegraphics[width=0.45\textwidth]{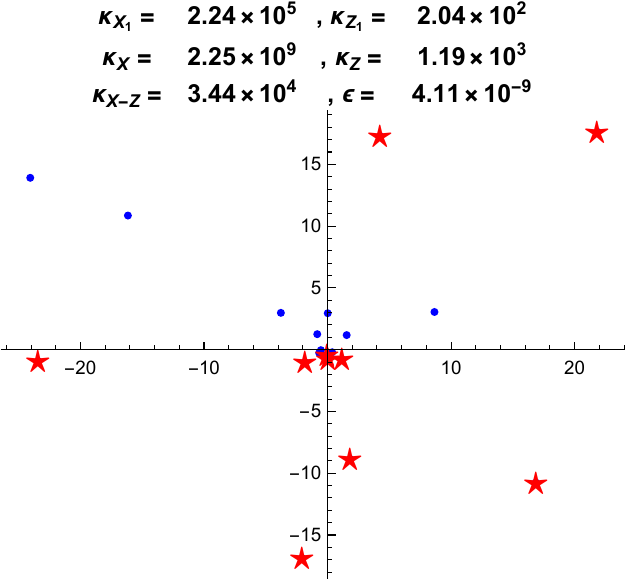}
 \caption{The best (left) and the worst (right) splittings of the spectrum of $\mathcal C_1$ from Example~\ref{ex:1}(c). The points of the spectrum of $\mathcal C_1$ related to $X$ are depicted by points, and related to $Z$ are depicted by stars} \label{f:comp10-10}
\end{center}
\end{figure}

\begin{example}\label{ex:1}
(a) We set $n=10$ and create matrices $B$ and $C$ consisting of random complex numbers uniformly distributed in $[-1,1]+[-i,i]$. We form the companion matrix $\mathcal C_1$ and find~\cite{Wolfram} its eigenvalues and eigenvectors. Since the matrices $B$ and $C$ are random, all eigenvalues are distinct and we obtain $2n=20$ different eigenvalues. We construct all splittings of the set $V$ of all normalized eigenvectors into two parts $V_X$ and $V_Z$, consisting of $n$ vectors. For each splitting, we create two block matrices
\begin{equation*}
\begin{pmatrix}
X_1 \\
X_2
\end{pmatrix},\qquad
\begin{pmatrix}
Z_1 \\
Z_2
\end{pmatrix},
\end{equation*}
whose columns are $V_X$ and $V_Z$ respectively.
Then we pass to the matrices (if $X_1$ or $Z_1$ is not invertible or very badly invertible we exclude this splitting from the consideration)
\begin{equation*}
 \begin{pmatrix} \mathbf1\\X\end{pmatrix}=
 \begin{pmatrix}
   X_1X_1^{-1} \\
   X_2X_1^{-1}
 \end{pmatrix},\qquad
 \begin{pmatrix} \mathbf1\\Z\end{pmatrix}=
 \begin{pmatrix}
  Z_1Z_1^{-1} \\
  Z_2Z_1^{-1}
 \end{pmatrix}.
\end{equation*}
Thus, we associate a complete pair of right solvents $X$ and $Z$ with each splitting (Theorem~\ref{t:2 roots}). For each pair, we calculate the condition numbers
\begin{equation*}
\varkappa(X_1),\;\varkappa(Z_1),\;\varkappa(X),\;\varkappa(Z),\;\varkappa(X-Z)
\end{equation*}
and the maximum $\varkappa_{\max}$ of these numbers. We select a complete pair of right solvents that corresponds to the smallest $\varkappa_{\max}$ and a complete pair of right solvents that corresponds to the largest $\varkappa_{\max}$; we call these pairs \emph{the best} and \emph{the worst} respectively.

We take $X_1$ and $X_2$ that correspond to the best pair. We create copies of $X_1$ and $X_2$ with 100 significant digits, adding zero digits at the end. We calculate $X$ and $Z$ again, and calculate the matrix (see Corollary~\ref{c:Perov})
\begin{equation*}
U(1)=(e^X-e^Z)(X-Z)^{-1}
\end{equation*}
with ordinary precision (about 16 significant digits) and with 100 significant digits; we denote the results by $U_{16}(1)$ and $U_{100}(1)$ respectively. We interpret $U_{100}(1)$ as the precise value. We calculate the \emph{relative error} of $U(1)$
\begin{equation*}
\varepsilon_{\text{best}}=\frac{\lVert U_{16}(1)-U_{100}(1)\rVert}{\lVert U_{16}(1)\rVert}
\end{equation*}
(for matrices, we use the norm induced by the Euclidian norm on $\mathbb C^n$).

We repeat the same calculations for the worst pair and denote the obtained number by $\varepsilon_{\text{worst}}$.
The result of calculations is presented in Figure~\ref{f:comp1-1}.
We have repeated this experiment several times; the results are similar.

We also perform two modifications of this experiment:
(b) with a matrix $B$ consisting of random complex numbers uniformly distributed in $[-1,1]+[-i,i]$ and a matrix $C$ consisting of random complex numbers uniformly distributed in $[-10,10]+[-10i,10i]$, see Figure~\ref{f:comp1-10};
(c) with matrices $B$ and $C$ consisting of random complex numbers uniformly distributed in $[-10,10]+[-10i,10i]$, see Figure~\ref{f:comp10-10}.
\end{example}

\begin{figure}[hbt]
\begin{center}
\includegraphics[width=0.45\textwidth]{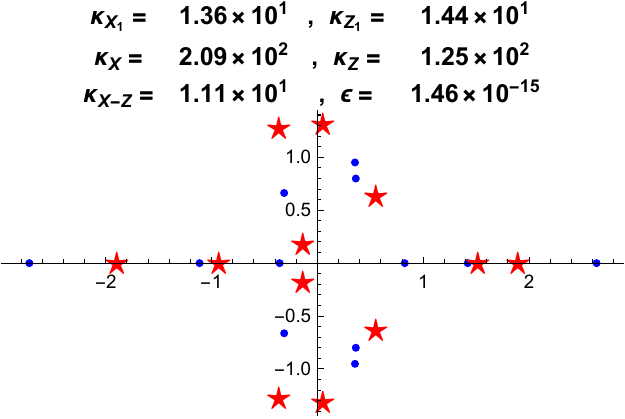}\hfill
\includegraphics[width=0.45\textwidth]{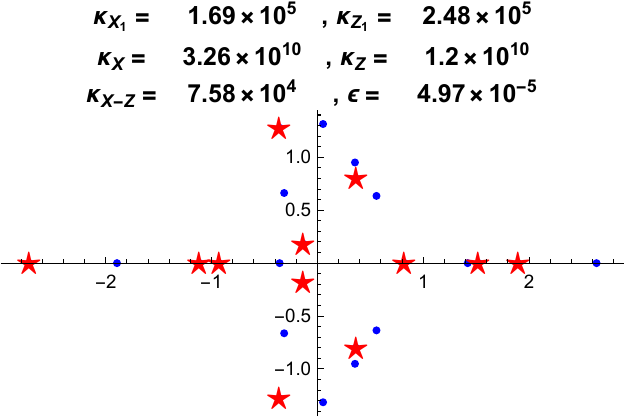}
 \caption{The best (left) and the worst (right) splittings of the spectrum of $\mathcal C_1$ from Example~\ref{ex:2}(a). The points of the spectrum of $\mathcal C_1$ related to $X$ are depicted by points, and related to $Z$ are depicted by stars} \label{f:real1-1}
\end{center}
\begin{center}
\includegraphics[width=0.45\textwidth]{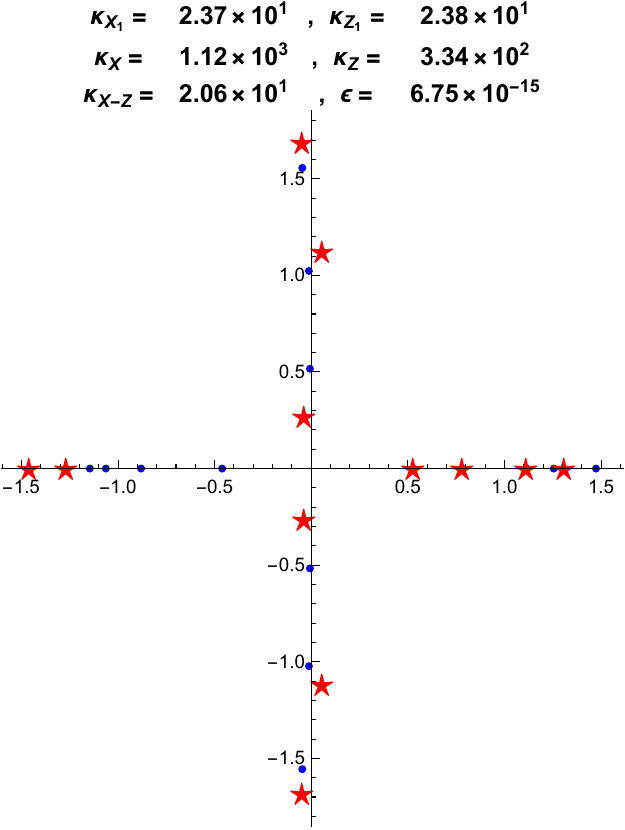}\hfill
\includegraphics[width=0.45\textwidth]{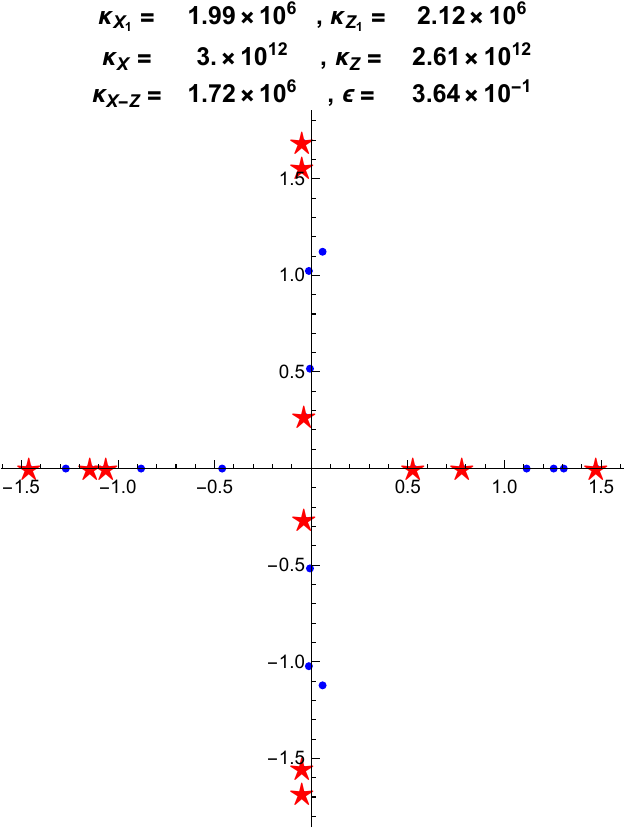}
 \caption{The best (left) and the worst (right) splittings of the spectrum of $\mathcal C_1$ from Example~\ref{ex:2}(b). The points of the spectrum of $\mathcal C_1$ related to $X$ are depicted by points, and related to $Z$ are depicted by stars} \label{f:real01-1}
\end{center}
%
\begin{center}
\includegraphics[width=0.45\textwidth]{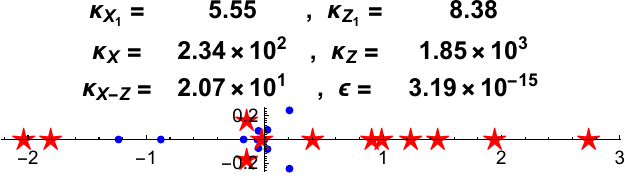}\hfill
\includegraphics[width=0.45\textwidth]{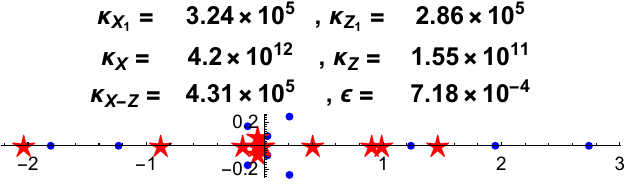}
 \caption{The best (left) and the worst (right) splittings of the spectrum of $\mathcal C_1$ from Example~\ref{ex:2}(c). The points of the spectrum of $\mathcal C_1$ related to $X$ are depicted by points, and related to $Z$ are depicted by stars } \label{f:real1-01}
\end{center}
\end{figure}

\begin{example}\label{ex:2}
We consider a real selfadjoint pencil (i.~e., the coefficients $B$ and $C$ are Hermitian matrices)~\cite{Keldysh51:eng,Keldysh71:eng,Krein-LangerI:eng,Krein-LangerII:eng,Langer67,Langer76,Marcus88:eng,
Shkalikov-MS88:eng,Shkalikov96:eng,Shkalikov98,Shkalikov-Pliev89:eng}.

(a) We set $n=12$ and create Hermitian matrices $B$ and $C$ (i.~e., $B^H=B$ and $C^H=C$) consisting of random real numbers uniformly distributed in $[-1,1]$. We form the companion matrix $\mathcal C_1$ and find its eigenvalues and eigenvectors. All eigenvalues are distinct and we obtain $2n=24$ different eigenvalues. We recall that non-real eigenvalues occur in complex conjugate pairs. We construct all splittings of the set $V$ of all normalized eigenvectors into two parts $V_X$ and $V_Z$, consisting of $n$ vectors, such that if a complex eigenvalue $\lambda$ is related to $V_X$, then the conjugate number $\bar\lambda$ is also related to $V_X$. Taking into account only such splittings ensures that the solvents $X$ and $Z$ are real.
After that we repeat the calculations from Example~\ref{ex:1}(a). The result is presented in Figure~\ref{f:real1-1}.

We perform two modifications of this experiment:
(b) with a matrix $B$ consisting of random real numbers uniformly distributed in $[-0.1,0.1]$ and a matrix $C$ consisting of random real numbers uniformly distributed in $[-1,1]$, see Figure~\ref{f:real01-1};
(c) with  a matrix $B$ consisting of random real numbers uniformly distributed in $[-1,1]$ and a matrix $C$ consisting of random real numbers uniformly distributed in $[-0.1,0.1]$, see Figure~\ref{f:real1-01}.
\end{example}

\begin{figure}[thb]
\begin{center}
\includegraphics[width=0.45\textwidth]{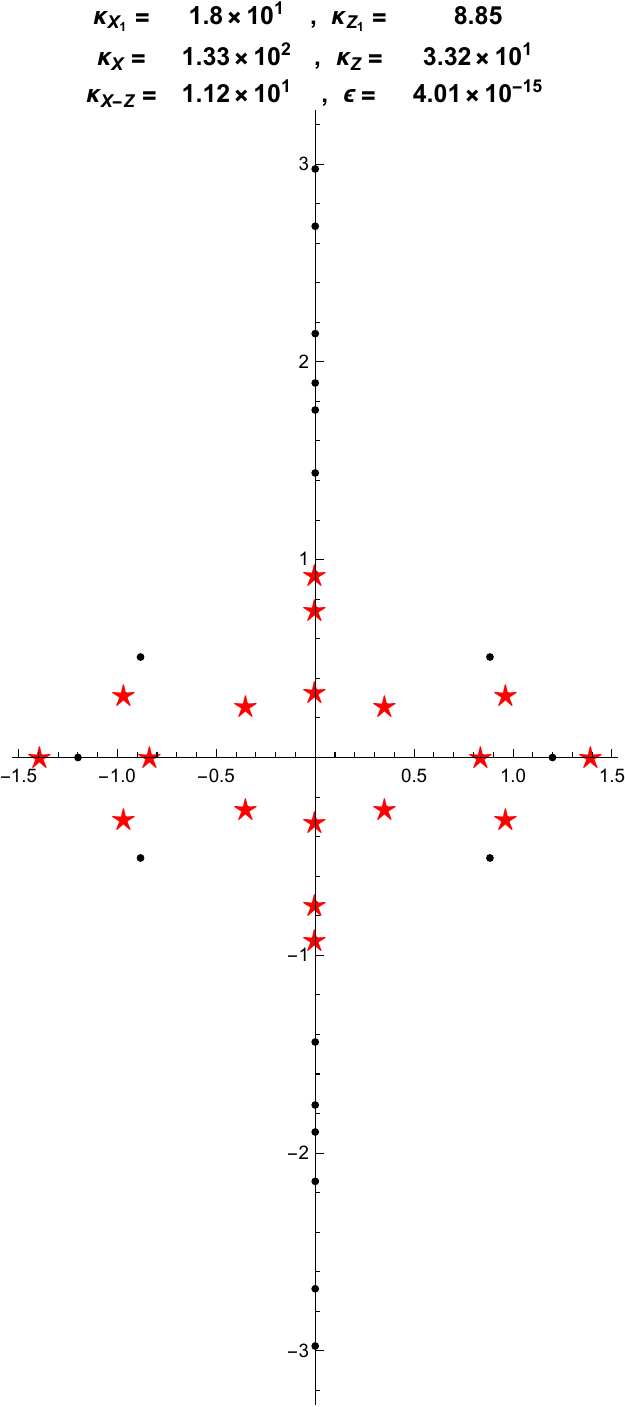}\hfill
\includegraphics[width=0.45\textwidth]{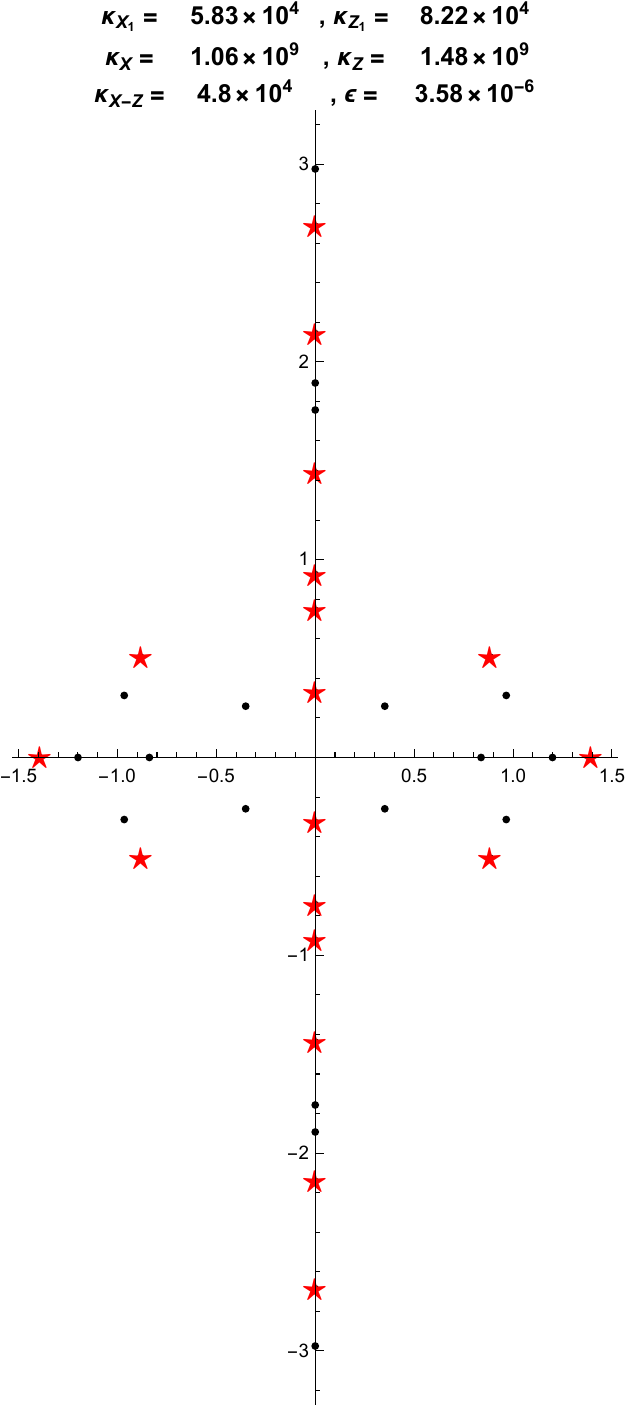}
 \caption{The best (left) and the worst (right) splittings of the spectrum of $\mathcal C_1$ from Example~\ref{ex:3}(a). The points of the spectrum of $\mathcal C_1$ related to $X$ are depicted by points, and related to $Z$ are depicted by stars} \label{f:gyro1-1}
\end{center}
\end{figure}
\begin{figure}[htb]
\begin{center}
\includegraphics[width=0.45\textwidth]{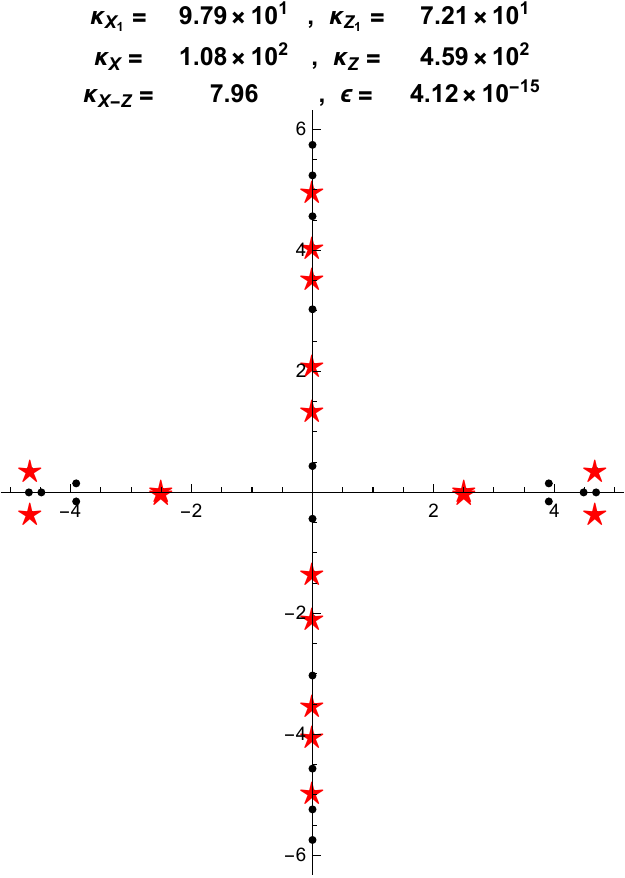}\hfill
\includegraphics[width=0.45\textwidth]{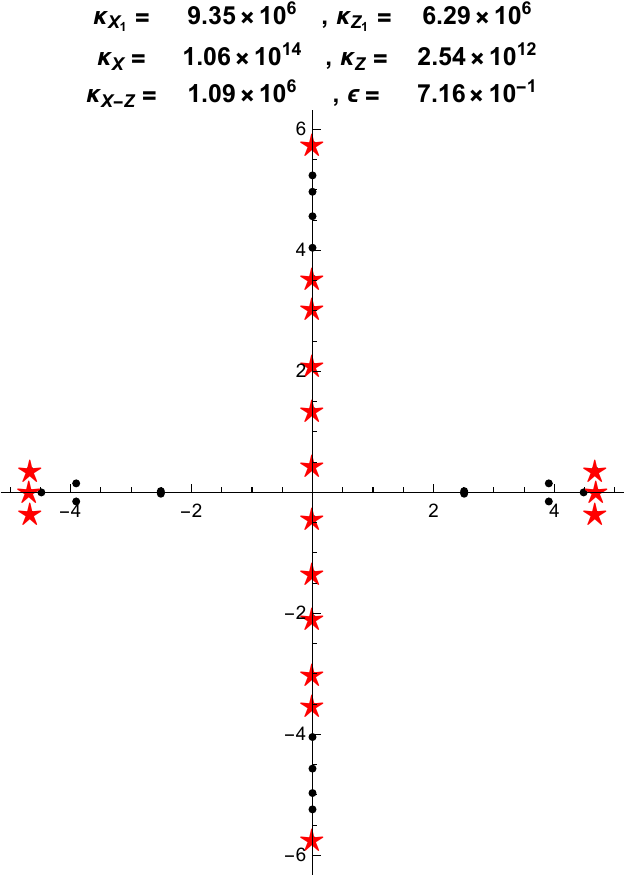}
 \caption{The best (left) and the worst (right) splittings of the spectrum of $\mathcal C_1$ from Example~\ref{ex:3}(b). The points of the spectrum of $\mathcal C_1$ related to $X$ are depicted by points, and related to $Z$ are depicted by stars} \label{f:gyro1-10}
\end{center}
\end{figure}

\begin{example}\label{ex:3}
We consider a pencil corresponding to canonical differential equation~\cite{Duffin60,Guo04,Lancaster66,
Lancaster99,Mehrmann-Watkin00,Tisseur-Meerbergen01},~\cite[ch.~II, \S~3.2]{Yakubovich-Starzhinskii75:eng}. Such equations describe gyroscopic systems. Namely, we assume that $B$ and $C$ are real, $C$ is Hermitian and $B$ is skew-Hermitian ($B^H=-B$). In this case, the spectrum of $\mathcal C_1$ is symmetric with respect to both the real and imaginary axes. So, it is natural to put all symmetric pairs of eigenvalues into the same parts.

(a) We set $n=18$ and create a skew-Hermitian matrix $B$ and a Hermitian matrix $C$ (i.~e., $B^H=-B$ and $C^H=C$) consisting of random real numbers uniformly distributed in $[-1,1]$. We form the companion matrix $\mathcal C_1$ and find its eigenvalues and eigenvectors. All eigenvalues are distinct and we obtain $2n=36$ different eigenvalues. We construct all splittings of the set $V$ of all normalized eigenvectors into two parts $V_X$ and $V_Z$, consisting of $n$ vectors, such that if an imaginary eigenvalue $\lambda$ is related to $V_X$, then the conjugate eigenvalue $\bar\lambda$ is also related to $V_X$, and if a real eigenvalue $\lambda$ is related to $V_X$, then the opposite eigenvalue $-\lambda$ is also related to $V_X$, and if a proper (i.~e., with both $\Real\lambda\neq0$ and $\Imaginary\lambda\neq0$) complex eigenvalue $\lambda$ is related to $V_X$, then the other three eigenvalues $-\lambda$, $\bar\lambda$, and $-\bar\lambda$ are also related to $V_X$.
We repeat the calculations from Example~\ref{ex:1}(a). The result is presented in Figure~\ref{f:gyro1-1}.

Then we perform a modification of this experiment:
(b) with a matrix $B$ consisting of random real numbers uniformly distributed in $[-1,1]$ and a matrix $C$ consisting of random real numbers uniformly distributed in $[-10,10]$, see Figure~\ref{f:gyro1-10}.
\end{example}

Numerical experiments carried out show that the relative error $\varepsilon_{\text{worst}}$ in $U(1)$ for the worst complete pair can be noticeably greater than the relative error $\varepsilon_{\text{best}}$ in $U(1)$ for best complete pair (especially, see Figures~\ref{f:real01-1} and~\ref{f:gyro1-10}). It is also worth noticing that the large relative error $\varepsilon$ is more likely associated with large $\varkappa(X)$ or $\varkappa(Z)$ than large $\varkappa(X-Z)$.
Some authors restrict themselves to searching for a dominant solvent, i.e., a solvent that corresponds to the largest eigenvalues in absolute value.
Figures~\ref{f:comp1-10},~\ref{f:comp10-10}, and~\ref{f:gyro1-1} confirm that
such a kind of solvents may lead to a good complete pair. Nevertheless, the remaining examples show that in general there is no simple connection between a good complete pair and a special type of splitting of the spectrum of $\mathcal C_1$ into two parts.

\end{document}